\documentclass{amsart} 

\usepackage{amssymb}
\usepackage{amsfonts}
\usepackage{amsmath}
\usepackage{mathrsfs}

\newcommand{\myself}{\author{Gianluca Cassese}
                     \address{Universit\`{a} Milano Bicocca and University of Lugano}
                     \email{gianluca.cassese@unimib.it}
                     \curraddr{Department of Statistics, Building U7, Room 239, via Bicocca 
                               degli Arcimboldi 8, 20126 Milano - Italy}}

\newtheorem{theorem}{Theorem}
\theoremstyle{plain}

\newtheorem{corollary}{Corollary}

\newtheorem{definition}{Definition}
\newtheorem{example}{Example}

\newtheorem{lemma}{Lemma}

\newtheorem{remark}{Remark}

\numberwithin{equation}{section}
\newcommand{\Prob}{\mathbb{P}} 
\newcommand{\LIM}{\mathrm{LIM}} 
 
\newcommand{\B}{\mathfrak{B}} 
\newcommand{\A}{\mathscr{A}} 
\newcommand{\F}{\mathscr{F}}
\newcommand{\R}{\mathbb{R}}

\newcommand{\restr}[2]{#1\vert #2} 
\newcommand{\lrestr}[2]{\left.#1\right\vert #2} 
\newcommand{\rrestr}[2]{#1\left\vert #2\right.} 
 
\newcommand{\abs}[1]{\vert #1\vert} 
\newcommand{\dabs}[1]{\left\vert #1\right\vert} 
\newcommand{\norm}[1]{\Vert #1\Vert} 
\newcommand{\dnorm}[1]{\left\Vert #1\right\Vert} 
\newcommand{\cond}[3]{#1(#2\vert #3)} 
 
\newcommand{\lcond}[3]{#1\left(\left.#2\right\vert #3\right)} 
\newcommand{\condP}[2]{P(#1\vert #2)} 
 
\newcommand{\lcondP}[2]{P\left(\left.#1\right\vert #2\right)}

\newcommand{\net}[3]{\left\langle #1_{#2}\right\rangle_{#2\in #3} } 
 
\newcommand{\seq}[2]{\net{#1}{#2}{\mathbb{N}}} 
 
\newcommand{\seqn}[1]{\seq{#1}{n}}

\newcommand{\M}{\mathscr M}
\newcommand{\Pred}{\mathscr{P}}

\newcommand{\Pd}{\bar\Pred^d}
\newcommand{\lP}{\lambda^P}
\newcommand{\ET}{\mathscr{E}}
\newcommand{\EU}{\bar{\ET}}
\newcommand{\D}{\mathscr D}
\newcommand{\DU}{\bar{\D}}
\newcommand{\La}{\mathscr L}
\newcommand{\C}{\mathscr C}
\newcommand{\BV}[1]{\dnorm{#1}_{V}}
\newcommand{\AAT}{(\A_t:t\in T)}
\newcommand{\mT}{(m_t:t\in T)}
\newcommand{\AU}{\bar\A}
\newcommand{\AAU}{\left(\AU_u:u\in U\right)}
\newcommand{\mU}{(\bar m_u:u\in U)}
\newcommand{\set}[1]{\mathbf{1}_{#1}}
\newcommand{\OI}{\bar\Omega}
\begin{document}
\title[Supermartingale Decomposition]{Supermartingale Decomposition 
with a General Index Set}
\myself
\date
\today

\subjclass[2000]{Primary .} 

\begin{abstract}
We prove results on the existence of Dol\'{e}ans-Dade measures and of the Doob-Meyer
decomposition for supermartingales indexed by a general index set.
\end{abstract}

\maketitle

\section{Introduction}

By Doob's theorem, supermartingales indexed by the natural numbers decompose into 
the difference of a uniformly integrable martingale and an increasing process. The 
relative ease of working with increasing processes explains the prominent role of 
this result in stochastic analysis and in the theory of stochastic integration. Meyer 
\cite{meyer} then proved that, under the \textit{usual conditions}, Doob's decomposition 
exists for right continuous supermartingales indexed by the positive reals if and 
only if the class $D$ property is satisfied. Dol\'{e}ans-Dade \cite{doleans-dade} 
was the first to represent supermartingales as measures over predictable rectangles 
and to prove that a supermartingale is of class $D$ if and only if its Dol\'{e}ans-Dade 
measure is countably additive. This line of approach has then become dominant in the 
work of authors such as F\"{o}llmer \cite{follmer} and Metivier and Pellaumail 
\cite{metivier pellaumail}. 

In this paper we cosider the case of processes indexed by general index sets, illustrated 
in the following
\begin{example} 
\label{example control}
Let $\C$ be a collection of $\F\otimes\mathcal B(\R_+)$ measurable functions
on $\Omega\times\R_+$ and $\mathscr U\subset\F\otimes\mathcal B(\R_+)$ an algebra. 
For each $u\in\mathscr U$, let 
$\mathscr B_u=\sigma\{\int_{v^c}f(c)dt:\ v\in\mathscr U,\ u\subset v,\ c\in\C\}$. 
Assume that $\C$ possesses the following property: $c,c^\circ\in\C$, $u\in\mathscr U$ and 
$F\in\mathscr B_u$ imply $c^\circ+(c-c^\circ)\set{F}\set{u^c}\in\C$. Let $c^*$ be a solution 
to the problem
\begin{equation}
\label{control}
\sup_{c\in\C}P\int f(c_t)dt
\end{equation}
If $c^\circ\in\C$ and $V_u=\int_{u^c}f(c^\circ)dt+\lcondP{\int_uf(c^*)dt}{\mathscr B_u}$ 
then $(V_u:u\in\mathscr U)$ is a supermartingale. In fact, $v\subset u$ implies 
$\mathscr B_u\subset\mathscr B_v$ and 
$\condP{V_v}{\mathscr B_u}=V_u+\lcondP{\int_{u\cap v^c}[f(c^\circ)-f(c^*)]dt}{\mathscr B_u}\leq V_u$.
\end{example}

We prove, in Theorem \ref{theorem Doleans}, a necessary and sufficient condition, the 
class $D_0$ property, for the existence of a Dol\'{e}ans-Dade measure associated to a 
supermartingale. Based on this, we establish, in Theorem \ref{theorem DM}, a
sufficient condition, the class $D_*$ property, for the existence of a Doob Meyer 
decomposition. In Corollary \ref{corollary DM uiv} we consider supermartingales
of uniformly integrable variation. The two key properties involved in our results are 
the possibility of extending the original supermartingale to a larger filtration and some 
form of the optional sampling theorem. From a mathematical perspective we exploit 
extensively results from the theory of finitely additive measures.

Among the many papers devoted to the theory of stochastic processes indexed by
general sets, the ones more directly related to ours are those of Dozzi, Ivanoff and 
Merzbach \cite{dozzi}, who obtain a form of the Doob Meyer decomposition, of Ivanoff 
and Merzbach \cite{ivanoff}, who extend such decomposition by a localization argument, 
and, to a much lesser extent, of De Giosa and Mininni \cite{de giosa}. These works, which draw 
in turn from an unpublished paper of Norberg \cite{norberg}, apply classical techniques, 
based on right continuity, separability and uniform integrability and, to make this 
possible, introduce a number of set-theoretical as well as of topological restrictions 
on the index set. These assumptions represent the main difference with the approach 
followed here\footnote{A more explicit comparison with this approach will be made in 
the following sections. See in particular Remark \ref{remark extension}.} which owes, 
perhaps, more to the work of Mertens \cite{mertens} than to that of Meyer \cite{meyer}.

\section{Preliminaries and notation}
We fix some general notation, mainly in accordance with \cite{bible}. When $S$ is a 
set $2^S$ denotes its power set and $\set{S}$ its indicator function. If 
$\Sigma\subset2^S$, typically an algebra, the symbols $ba(\Sigma)$ (resp. $ca(\Sigma)$) 
and $\B(\Sigma)$ designate the spaces of finitely (resp. countably) additive set 
functions on $\Sigma$ and the closure of the set of $\Sigma$ simple functions with 
respect to the supremum norm, respectively (however we prefer $\B(S)$ to $\B(2^S)$). 
$\Prob(\Sigma)$ will be the subcollection of $ba(S)_+$ consisting of those elements $Q$
whose restriction $\restr{Q}{\Sigma}$ to $\Sigma$ is a countably additive probability measure. 
Finitely additive measures will throughout be identified with the linear 
functional arising from the corresponding expected value so that we prefer writing 
$\mu(f)$ rather than $\int fd\mu$. We recall that if $\Sigma$ is an algebra of subsets 
of $S$ and $\mu\in ba(\Sigma)_+$ then there exists $\bar\mu\in ba(2^S)_+$ such that 
$\restr{\bar\mu}{\Sigma}=\mu$, \cite[theorem 3.2.10, p. 70]{rao}.

We take two sets $\Omega$ and $I$ as given, put $\OI\equiv\Omega\times I$ and write 
$ba(\OI)$ and $\B(\OI)$ more briefly as $ba$ and $\B$. $2^{\OI}$ is ordered by reverse 
inclusion that is $s\leq t$ whenever $t\subset s$; $s<t$ means $s\leq t$ and 
$s\cap t^c\neq\varnothing$. $t(i)$ denotes the $i$-section 
$\{\omega\in\Omega:(\omega,i)\in t\}$ of $t$ and 
$\{s<t\}=\bigcup_{i\in I}(s\cap t^c)(i)$: thus $\{s<\varnothing\}$ is 
just the projection of $s$ on $\Omega$. The special case where $I=\R_+$ and some
probability measure $P$ on $\Omega$ is given will be referred to as the classical theory.

Also given are a collection $T$ of subsets of $\OI$ containing $\{\OI\}$ and a filtration 
$\Bbb A=\AAT$, that is an increasing collection of algebras of subsets of $\Omega$ satisfying:
\begin{equation}
\label{filtration}
F\cap (s\cap t^c)(i)\in\A_s\cap\A_t\quad\text{and}\quad F\cap \{s<t\}\in\A_s\cap\A_t
\qquad s,t\in T,\ F\in\A_s,\ i\in I
\end{equation} 
Define also $\A=\bigcup_{t\in T}\A_t$ and $\F=\sigma\A$. 

In the classical theory, $T$ would tipically be a family of stochastic intervals
such as $]]\tau,\infty[[$ or $[[\tau,\infty[[$ with $\tau$ a stopping time; the
case where each $t\in T$ is deterministic, i.e. of the form $t=\Omega\times J$ with
$J\subset I$, is of course a possibility. Dozzi et al. \cite{dozzi} take $T$ to be a 
lattice of closed subsets of a (locally) compact topological space and assume that $T$ is 
closed with respect to countable intersections. One should remark that in the present 
setting the second inclusion in (\ref{filtration}) does not follow from the first one 
and must therefore be explicitly assumed. 

Repeatedly in what follows we shall take into consideration an alternative filtration
$\bar{\Bbb A}=\AAU$ where $U\subset2^{\OI}$ is closed under union and intersection, 
$T\cup\{\varnothing\}\subset U$ and $\A_t\subset\AU_t$ for all $t\in T$. As a matter 
of notation, the same symbol denoting some object defined relatively to $\Bbb A$ will 
be used with an upper bar to designate the corresponding object defined relatively to 
$\bar{\Bbb A}$.

\section{Finitely Additive Supermartingales} 
\label{sec finitely additive}

A finitely additive process (on $\Bbb A$) is an element $m=\mT$ of the product space 
$\prod_{t\in T}ba(\A_t)$. We shall always use the convention of letting 
$m_\varnothing$ be the null measure on $\A$. A finitely additive process $m$ 
is bounded if $\norm{m}\equiv\sup_{t\in T}\norm{m_t}<\infty$; it is of bounded 
variation if
\begin{equation}
\label{BV}
\BV{m}\equiv\sup\sum_{n=1}^N\dnorm{m_{s_n}-\rrestr{m_{t_n}}{\A_{s_n}}}<\infty
\end{equation} 
the supremum being over the family $\D$ of all finite, disjoint collections 
\begin{equation}
\label{delta}
d=\{s_n\cap t^c_n:n=1,\ldots,N\}\quad\text{with}\quad s_n,t_n\in T\cup\{\varnothing\},\ 
s_n\leq t_n\qquad n=1,\ldots,N
\end{equation}
Processes of bounded variation are 
thus bounded. Our definition (\ref{BV}) slightly departs from the original 
one of Fisk \cite{fisk} as we do not require $t_n\leq s_{n+1}$ for all $n$.

We speak of the finitely additive process $m$ as a finitely additive 
supermartingale if $\restr{m_t}{\A_s}\leq m_s$ 
for all $s,t\in T$ such that $s\leq t$. $f:\OI\to\R$ is an elementary process, 
$f\in\ET$, if it may be written in the form
\begin{equation}
\label{elementary}
f=\sum_{n=1}^Nf_n\set{t_n}\quad\text{with}\quad f_n\in\B(\A_{t_n}),\ 
t_n\in T\qquad n=1,\ldots,N
\end{equation}
$\Pred$ denotes the (predictable) $\sigma$ algebra generated by the elementary 
processes. We write $f\in\ET^*$ if the requirement $f_n\in\B(\A_{t_n})$ in 
(\ref{elementary}) is replaced by $f_n\in\B(\F)$ for $n=1,\ldots,N$.

A finitely additive supermartingale $m$ is strong if it is of bounded variation and
\begin{equation}
\label{strong}
0=\sum_{n=1}^Nf_n\set{t_n} \in\ET\qquad\text{imply}\qquad
\sum_{n=1}^Nm_{t_n}(f_n)=0
\end{equation}
As $\set{\{s<t\}^c}\set{s\cap t^c}=0$, a strong finitely additive 
supermartingale $m$ must satisfy $m_s(\{s<t\}^c)=m_{t}(\{s<t\}^c)$ for all $s\leq t$: 
implicit in (\ref{strong}) is thus a version of the optional sampling 
theorem. It is known that this theorem is far from obvious with a general 
index set (see \cite{hurzeler} and \cite{kurtz}) and it may actually fail even with
$\R_+$ as the index set unless the usual conditions hold (see 
\cite[p. 393]{dellacherie meyer}, from which our terminology is borrowed). 
The assumption that a finitely additive supermartingale is strong will thus
play a major role. 

As argued by Dozzi et al. \cite{dozzi}, for many a purpose the index set 
$T$ does not possess enough mathematical structure, a problem that induces 
to consider possible extensions. The following example illustrates that this
may not be entirely obvious.

\begin{example}
\label{example extension}
Let $\bar m=\mU$ be a finitely additive supermartingale on $\bar{\Bbb A}$ and 
define the semi-algebra
\begin{equation}
\label{semi ring}
U(d)=\{u\cap v^c:u,v\in U,\ u\leq v\} 
\end{equation}
For each $u\cap v^c\in U(d)$ one may be tempted to extend $\bar m$ to $U(d)$ by letting
\begin{equation}
\label{norberg}
\bar m_{u\cap v^c}=(\bar m_u-\restr{\bar m_v}{\AU_u})\in ba(\AU_u)
\end{equation}
Such an extension however need not be well defined if $\bar m$ is not strong.
Consider in fact $u_1,v_1,u_2,v_2\in U$ such that $u_1\leq v_1\leq u_2\leq v_2$ and 
$\bar m_{u_1}(\Omega)>\bar m_{v_1}(\Omega)>\bar m_{u_2}(\Omega)=\bar m_{v_2}(\Omega)$. 
If $u_1\cap v_1^c=u_2\cap v_2^c$ then (\ref{norberg}) implies 
$\restr{(\bar m_{u_1}-\bar m_{v_1})}{\AU_{u_1}}=\restr{(\bar m_{u_2}-\bar m_{v_2})}{\AU_{u_1}}=0$
which contradicts $(\bar m_{u_1}-\bar m_{v_1})(\Omega)>0$. In the case of classical 
processes, when each $\bar m_u\in ba(\AU_u)$ is replaced by a corresponding $\sigma\AU_u$ measurable 
random quantity $X_u$, assume that $X_{u_2}=X_{v_2}$. Then (\ref{norberg}) implies 
$X_{u_1}=X_{v_1}$ which is contradictory if $X_{v_1}$ fails to be $\sigma\AU_{u_1}$ measurable.
\end{example}

\begin{remark}
\label{remark extension}
In many papers (see, e.g.,  \cite[p. 74]{de giosa}, \cite[proposition 2.1]{dozzi} 
or \cite[p. 85]{ivanoff}) the extension (\ref{norberg}) of a process indexed by a 
lattice of sets to its generated semi-ring is claimed to exist by virtue of 
\cite[proposition 2.3, p. 9]{norberg}. This claim is however not correct for all 
lattices and all processes, as Example \ref{example extension} shows. In fact 
in the paper by Norberg, the index set consists of all lower sets 
$\downarrow f\equiv\{g\in\mathscr L:g\leq f\}$ of elements $f$ of 
some lattice $\mathscr L$ satisfying the property\footnote{The proof of 
\cite[lemma 2.2 ]{norberg} uses property (\ref{antisymmetric}) without explicitly 
assuming it.}
\begin{equation}
\label{antisymmetric}
f\leq g\quad\text{and}\quad g\leq f\quad\text{imply}\quad f=g,\qquad f,g\in\mathscr L
\end{equation}
In this setting if $f_i,g_i\in\mathscr L$ and $g_i\leq f_i$ for $i=1,2$ 
then $\downarrow f_1\cap(\downarrow g_1)^c=\downarrow f_2\cap(\downarrow g_2)^c$
if and only if either $f_i=g_i$ for $i=1,2$ or $f_1=f_2$ and $g_1=g_2$. Thus
Example \ref{example extension} does not apply.
\end{remark}

Implicit in the above remarks is the importance of the set-theoretic properties 
of the index set. This is specially true for what concerns the representation of 
elementary processes. Let, for example, $f=\sum_{k=1}^Kf'_k\set{u'_k}\in\EU$. 
Denote by $\{\pi_1,\ldots,\pi_N\}$ the collection of non empty subsets of 
$\{1,\ldots,K\}$ and, for $n=1,\ldots,N$, define (with the convention 
$\bigcup\varnothing=\varnothing$)
$$
u_n=\bigcap_{k\in\pi_n}u'_k,\quad v_n=u_n\cap\bigcup_{j\notin\pi_n}u'_j 
\quad\text{and}\quad f_n=\sum_{k\in\pi_n}f'_k
$$
Then
\begin{equation}
\label{canonical}
f=\sum_{n=1}^Nf_n\set{u_n\cap v_n^c}\quad\text{where}\quad f_n\in\mathscr S(\AU_{u_n})\ 
n=1,\ldots,N\quad\text{and}\quad\{u_n\cap v_n^c:n=1,\ldots,N\}\in\DU
\end{equation}
(see (\ref{delta}) for the definition of $\DU$). We will henceforth refer to 
(\ref{canonical}) as the \textit{canonical representation} of $f$. Another noteworthy 
implication of the set theoretic properties of $U$ is that the collection $\DU$ defined 
as in (\ref{delta}) is a directed set relatively to refinement, that is if we write 
$\delta'\geq\delta$ whenever $\delta,\delta'\in\DU$ and each $u\cap v^c\in\delta$ may 
be written as $\bigcup_{n=1}^Nu_n\cap v_n^c$ with $u_n\cap v_n^c\in\delta'$ $n=1,\ldots,N$.

Eventually we have the following (with $f^+=f\vee 0$):
\begin{lemma} 
\label{lemma lattice}
Let $\bar m$ be a positive, strong, finitely additive supermartingale on $\bar{\Bbb A}$. Then
\begin{equation}
\label{domination}
\sum_{n=1}^N\bar m_{u'_n}(f'_n)\leq\norm{\bar m}_V\dnorm{f^+}\qquad f=\sum_{n=1}^Nf'_n\set{u'_n}\in\EU
\end{equation}
\end{lemma}

\begin{proof} 
Fix $f=\sum_{n=1}^Nf'_n\set{u'_n}\in\EU$ and let $\sum_{k=1}^Kf_k\set{u_k\cap v^c_k}$
be its canonical representation as in (\ref{canonical}). For $k=1,\ldots,K$ let 
$\mu_{f,k}\in ba_+$ be an extension of $\bar m_{u_k}-\restr{\bar m_{v_k}}{\AU_{u_k}}$ 
to $2^\Omega$. Put $\mu_f=\sum_{k=1}^K\mu_{f,k}$: the inequality 
$\norm{\mu_f}\leq\norm{\bar m}_V$ is obious. Thus
$$
\sum_{n=1}^N\bar m_{u'_n}(f'_n)=\sum_{k=1}^K (\bar m_{u_k}-\bar m_{v_k})(f_k\set{\{u_k<v_k\} })
=\sum_{k=1}^K \mu_{f,k}(f_k\set{\{u_k<v_k\} })
\leq\mu_f\left(\sup_{i\in I}f(i)\right)
$$
the first equality being a consequence of the fact that $\bar m$ is strong. 
\end{proof}

Lemma \ref{lemma lattice} legitimates a special interest for conditions that permit 
the construction of an extension of our original finitely additive supermartingale $m$, 
that is of a finitely additive process $\bar m$ defined on $\bar{\Bbb A}$ such that 
$\restr{\bar m_t}{\A_t}=m_t$ for all $t\in T$ (in symbols $\restr{\bar m}{\Bbb A}=m$). 
It turns out that the problem of extending $m$ is related to the time honoured question 
of whether finitely additive supermartingales may be represented as measures on $\OI$, 
i.e. the existence of Dol\'{e}ans-Dade measures. 

\begin{theorem}
\label{theorem Doleans}
Let $m$ be a finitely additive process on $\Bbb A$. The following are equivalent:
\begin{enumerate}
\item[(\textit{i})] $m$ is of class $D_0$, that is 
\begin{equation}
\label{D0}
\sup\left\{\sum_{n=1}^Nm_{t_n}(f_n):1\geq\sum_{n=1}^Nf_n\set{t_n}\in\ET\right\}<\infty
\end{equation}
\item[(\textit{ii})] $m$ admits a Dol\'{e}ans-Dade measure, that is an 
element of the set 
$$
\M(m)=\left\{\lambda\in ba_+:\lambda\left(F\set{t}\right)=m_t(F),\ F\in\A_t,\ t\in T\right\}
$$
\item[(\textit{iii})] there exists a positive, strong, finitely additive 
supermartingale $\bar m$ on $\bar{\Bbb A}$ such that $\restr{\bar m}{\Bbb A}=m$.
\end{enumerate}
\end{theorem}

\begin{proof}
For $t\in T$, let $\La_t=\{f\set{t}:f\in\B(\A_t)\}$ and define 
$\phi_t:\La_t\rightarrow\R$ implicitly as $\phi_t(f\set{t})=m_t(f)$. Then 
$\La_t$ is a linear subspace of $\B$ and $\phi_t$ a linear functional on it. 
Condition (\textit{i}) corresponds to requiring that the collection 
$(\phi_t:t\in T)$ is coherent in the sense of \cite[Corollary 1]{wins}: 
thus (\textit{i}) is equivalent to (\textit{ii}). For $\lambda\in\M(m)$, define 
$\bar m\in\prod_{u\in U} ba(\AU_u)$ implicitly as 
\begin{equation}
\label{extension}
\bar m_u(F)=\lambda(F\set{u})\qquad u\in U,\ F\in\AU_u
\end{equation}
It is clear that $\bar m$ is a positive, strong, finitely additive supermartingale 
and that $\restr{\bar m}{\Bbb A}=m$. The implication (\textit{iii})$\to$(\textit{i}) 
follows from Lemma \ref{lemma lattice}.
\end{proof}

A Dol\'{e}ans-Dade measure is usually defined to be countably additive 
relatively to $\mathscr P$ (see \cite{doleans-dade}, \cite{follmer} and 
\cite{metivier pellaumail}). We believe that this additional property is 
not really essential to obtain several interesting results, such as the 
Doob Meyer decomposition. The advantage of taking $\M(m)$ to be a subset 
of $ba$, implicit in Theorem \ref{theorem Doleans}, is that the existence 
of a measure so defined turns out to be a property independent of the given 
filtration\footnote{By a terminological \textit{curiosum}, a finitely additive 
supermartingale as in Theorem \ref{theorem Doleans}.(\textit{iii}) is called 
a weak supermartingale in \cite[definition 2.1, p. 503]{dozzi}.}. 
It should be mentioned that a strong finitely additive supermartingale need
not be of class $D_0$, if the index set is not closed with respect to union 
and intersection.

The next task is to establish a version of the Doob Meyer decomposition. A 
first step in this direction is made by remarking that, if $\lambda\in\M(m)$ 
and $H\subset\OI$, we may define $\lambda_H\in ba(\Omega)$ implicitly as 
\begin{equation}
\label{lambda_H}
\lambda_H(F)=\lambda((F\times I)\cap H)\qquad F\subset\Omega
\end{equation}
Then $(\lambda_{u^c}:u\in U)$ is an increasing family in $ba(\Omega)$; moreover
\begin{equation}
\label{ADM}
m_t=\restr{\lambda_{\OI}}{\A_t}-\restr{\lambda_{t^c}}{\A_t}\qquad t\in T
\end{equation}
i.e. $m$ decomposes into the difference of a finitely additive
martingale and a finitely additive increasing process, a result first obtained
by Armstrong \cite{armstrong} (see also \cite[corollary 1, p. 591]{JOTP}).

\section{Classical Supermartingales}

A particularly interesting special case is that of classical supermartingales that
we treat, in accordance with \cite{JOTP}, without the assumption of a given probability 
measure. In order to avoid additional notation we assume in what follows (and without 
loss of generality) that $\AU_u$ is a $\sigma$ algebra for each $u\in U$.

Let $m$ be a positive finitely additive supermartingale on $\Bbb A$ and define
\begin{equation} 
\label{uc}
\M^{uc}=\{\lambda\in ba_+:\restr{\lambda_{\OI}}{\F}\in ca(\F)\}
\quad\text{and}\quad \M^{uc}(m)=\M^{uc}\cap\M(m)
\end{equation}
The property $\M^{uc}(m)\neq\varnothing$ is actually necessary and sufficient 
for a finitely additive supermartingale indexed by $\R_+$ to admit a Doob Meyer 
decomposition, \cite[theorem 4, p. 597]{JOTP}. This conclusion extends to the 
case of a linearly ordered index set but needs to be strengthened considerably 
in the more general case. It suffices, however, to get a primitive version of 
this fundamental result. Let in fact $\lambda\in\M^{uc}(m)$ and fix $P\in\Prob(\F)$ 
such that $\restr{\lambda_{\OI}}{\F}\ll\restr{P}{\F}$ (in symbols $P\in\Prob(\F,\lambda)$). 
Then, letting $\bar M,\ \bar A_u,\ \bar X_u\in L(P)$ be the Radon Nikodym derivatives 
of $\restr{\lambda_{\OI}}{\F},\ \restr{\lambda_{u^c}}{\F}$ and $\restr{\lambda_u}{\F}$ 
with respect to $\restr{P}{\F}$ we get from (\ref{ADM}):
\begin{equation}
\label{DM raw}
\bar m_u(F)=P(\bar X_u\set{F})=P((\bar M-\bar A_u)\set{F})\qquad u\in U,\ F\in\AU_u
\end{equation}
where $\bar m$ is the extension of $m$ to $\bar{\Bbb A}$ mentioned in Theorem 
\ref{theorem Doleans}.(\textit{iii}). The crucial step is thus obtaining a version
of (\ref{DM raw}) in which $\bar A$ is predictable in some due sense.

In the classical theory, where $P\in\Prob(\F)$ is given, there is a strict relationship
between the property $\M^{uc}\neq\varnothing$ and the existence of a Dol\'{e}ans-Dade
measure which is countably additive in restriction to the predictable $\sigma$ algebra.
In fact ordinary properties of the predictable projection guarantee that if $\lambda\in ba$ 
that vanishes on $P$ negligible sets and is countably additive in restriction to $\Pred$ 
then it has an extension which is countably additive on $\F\otimes\mathcal B(\R_+)$ and 
thus on $\F$. Conversely, each $\lambda\in\M^{uc}$ is a Daniell integral in restriction 
to any vector sublattice $\La$ of $\B$ with the property that $\sup_{i\in I}f(i)$ is $\F$ 
measurable for each $f\in\La$ and that each sequence $\seqn{f}$ in $\La$ with $f_n\downarrow 0$ 
converges to $0$ uniformly in $i\in I$. With $I=\R_+$ a simple variant of Dini's theorem 
shows that processes with finite range and upperly semicontinuous paths possess this 
property, a fact used by Dellacherie and Meyer \cite[theorem 2, p. 184]{dellacherie meyer} 
to prove, after Mertens \cite{mertens}, a version of the Doob Meyer decomposition that 
does not require the usual conditions. Under the current assumptions, however, this last 
argument does not hold without introducing additional topological properties; on the 
other hand, the extension of the notion of predictable projection to our setting, that 
in Dozzi et al. \cite[A3, p. 516] {dozzi} is by assumption, is not obvious.

For given $P\in\Prob(\F)$ and $d\in\DU$ define the following elementary process (up 
to a $P$ equivalence class)
\begin{equation}
\label{Pd}
\Pd_P(b)=\sum_{u\cap v^c\in d}\lcond{P}{\inf_{i\in u\cap v^c}b(i)}{\AU_u}\set{u\cap v^c}
\qquad b\in\B
\end{equation}
where the conditional expectation appearing in (\ref{Pd}) is defined as in 
\cite[Theorem 1, p. 588]{JOTP}. Of course if $b\in\EU$ and $d\in\DU$ is large enough
then $\Pd_P(b)=b$.

\begin{lemma}
\label{lemma predictable}
If $\lambda\in\M^{uc}$ and $P\in\Prob(\F,\lambda)$ then there is $\lP\in ba_+$ such that 
$\lP_{\OI}$ vanishes on $P$ null sets and\footnote{Where $\LIM$ denotes the Banach limit.}
\begin{equation}
\label{compensator}
\lP(fg)=\LIM_{d\in\DU}\lP\left(\Pd_P(f)g\right)\qquad f\in\EU^*,\ g\in\EU
\end{equation}
If $\lrestr{\lP}{\bar\Pred}\in ca(\bar\Pred)$ then there exists 
$\bar\Pred_P:\B\to L^\infty(\restr{\lambda}{\bar\Pred})$ such that
\begin{equation}
\label{exact compensator}
\bar\Pred_P(g)=g\quad\text{and}\quad\lP(f)=\lP(\bar\Pred_P(f))\qquad f\in\EU^*,
\ g\in\EU
\end{equation}
\end{lemma}

\begin{proof}
Consider the functional $\gamma:\B\to\R$ defined as 
$\gamma(f)=\LIM_{d\in\D}\lambda\left(\Pd_P(f)\right)$ for any $f\in\B$. Then, 
$\gamma$ is a concave integral in the sense of \cite[definition 1]{wins}, it 
is linear on $\EU^*$ and such that $\gamma=\lambda$ in restriction to $\EU$; 
moreover, $\gamma(b)=0$ for all $b$ in the linear space 
$\mathscr L=\{g\in\B:P(\abs{g}>\eta)=0\ \text{for all}\ \eta>0\}$. 
Given that $\mathscr L$ is a linear space, by \cite[lemma 2, p. 4]{wins} there exists 
$\lP\in ba_+$ such that 
$$
\lP(g)=0\quad\text{and}\quad\lP(f)\geq\gamma(f)\qquad g\in\mathscr L,\ f\in\B
$$
Consequently if $g\in\EU$ and $f\in\EU^*$, then $fg\in\EU^*$ and thus
$$
\lP(fg)=\gamma(fg)=\LIM_{d\in\DU}\lambda(\Pd_P(fg))=
\LIM_{d\in\DU}\lambda(\Pd_P(f)g)=\LIM_{d\in\DU}\lP(\Pd_P(f)g)
$$
The last claim follows by taking $\bar\Pred_P(f)=\cond{\lP}{f}{\bar\Pred}$.
\end{proof}

When $\lambda\in\M^{uc}$ and $P\in\Prob(\F,\lambda)$ then $\lP\in ba_+$, defined 
as in Lemma \ref{lemma predictable}, will be referred to as the $P$ compensator of 
$\lambda$, disregarding non uniqueness. Remark that if $\lambda\in\M^{uc}(m)$
for some finitely additive supermartingale $m$, then its $P$ compensator
$\lP$ is itself a Dol\'{e}ans-Dade measure for $m$, i.e. $\lP\in\M(m)$. It is,
however, not possible to conclude that $\lP\in\M^{uc}(m)$ in the general case, i.e. that
\begin{equation}
\label{class D*}
\M^*(m)=\left\{\lP\in\M^{uc}(m):\lambda\in\M^{uc}(m)\right\}\neq\varnothing
\end{equation}
We shall refer to (\ref{class D*}) by saying that $m$ is of class $D_*$. 

When $P\in\Prob(\F)$, $(\bar A_u:u\in U)$ is a $P$ increasing process if 
$P(0=\bar A_{\OI}\leq\bar A_u\leq\bar A_v)=1$ for all $u,v\in U$ with $u\leq v$. 
$(\bar B_u:u\in U)$ is then a modification of $\bar A$ if $P(\bar A_u=\bar B_u)=1$ 
for all $u\in U$.

\begin{theorem}
\label{theorem DM}
Let $\bar m$ be a finitely additive supermartingale of class $D_*$ on $\bar{\Bbb A}$. 
Then for some $P\in\Prob(\F)$ there exists one and only one (up to modification) 
way of writing
\begin{equation}
\label{DM}
\bar m_u(F)=P((M-\bar A_u)\set{F})\qquad u\in U,\ F\in\AU_u
\end{equation}
where $M\in L(P)$ and $\bar A$ is an increasing process, adapted to $\bar{\Bbb A}$ 
and such that
\begin{equation}
\label{natural}
P\int fd\bar A=\LIM_{d\in\DU}P\int_{u^c}\bar\Pred^d_P(f)d\bar A\qquad f\in\EU^*
\end{equation}
\end{theorem}

\begin{proof} 
Let $\lambda\in\M^{uc}(m)$, $P\in\Prob(\F,\lambda)$ and $\lambda^P\in\M^*(m)$. 
Define $M$ and $\bar A'_u$ to be  the Radon Nikodym derivatives with respect to 
$\restr{P}{\F}$ of $\restr{\lambda^P_{\OI}}{\F}$  and $\restr{\lambda_{u^c}^P}{\F}$. 
(\ref{DM}) is thus a version of (\ref{DM raw}). Clearly ,
\begin{equation}
P\int fd\bar A'=P\sum_{u\cap v^c\in d}f_u(\bar A'_v-\bar A'_u)=\lambda^P(f)\quad\text{for all}\quad
f=\sum_{u\cap v^c\in d}f_u\set{u\cap v^c}\in\EU^*
\end{equation}
(\ref{compensator}) then implies that (\ref{natural}) holds for $\bar A'$ and
its modifications, among which there exists one which is adapted. In fact, if 
$b\in\B(\F)$, $\bar u\in U$ and $\condP{b}{\AU_{\bar u}}=0$, then, choosing 
$d\in\DU$ finer than $\{\bar u^c\}$
$$
\Pd_P(b\set{\bar u^c})=\sum_{\{u\cap v^c\in d:u\cap v^c\subset\bar u^c\}}\lcondP{b}{\AU_{u}}\set{u\cap v^c}
=\sum_{\{u\cap v^c\in d:u\cap v^c\subset\bar u^c\}}\lcondP{b\set{\{u\leq\bar u\}}}{\AU_{u}}\set{u\cap v^c}=0
$$
a conclusion following from (\ref{filtration}) and the fact that 
$\lcondP{b\set{\{u\leq\bar u\}}}{\AU_{u}}=\lcondP{\condP{b}{\AU_{\bar u}}\set{\{u\leq\bar u\}}}{\AU_{u}}$. 
We conclude that $\lambda_{\bar u^c}^P(b)=\lambda^P(b\set{\bar u^c})=0$. Let 
$\bar A_u=\condP{\bar A'_u}{\AU_u}$ and $F\in\F$. Then,
\begin{equation}
P(\bar A'_u\set{F})=\lambda_{u^c}^P(F)=\lambda_{u^c}^P\left(\condP{F}{\AU_u}\right)
        =P(\bar A'_u\condP{F}{\AU_u})
        =P(\bar A_u\set{F})
\end{equation}
$\bar A$ is thus an adapted modification of $\bar A'$ and therefore itself an increasing process 
meeting (\ref{natural}). Suppose that $\condP{N}{\AU_u}-\bar B_u$ is another decomposition such 
as (\ref{DM}). Then if $F\in\AU_u$ and $d\in\DU$
$$
P\int\Pd_P(F\set{u})d\bar A=-P\int\Pd_P(F\set{u})d\bar X=P\int\Pd_P(F\set{u})d\bar B
$$
and, if both $\bar A$ and $\bar B$ meet (\ref{natural}), $P(\bar A_u\set{F})=P(\bar B_u\set{F})$.
\end{proof}

Remark that Theorem \ref{theorem DM} is actually weaker than the classical Doob Meyer 
decomposition first of all because the class $D_*$ property is only a sufficient condition but 
need not be necessary. Second, we established uniqueness only up to a modification rather 
than indistinguishability, a circumstance which is almost unavoidable in the absence of 
separability of the index set and of right continuity of the process.

It should also be remarked that it may not be possible to establish the above decomposition 
for the original index set $T$ because the increasing process $\bar A$ may not be adapted 
to the \textit{original} filtration $\Bbb A$. On the other hand the class $D_*$ property is a 
global property and is thus preserved under any enlargement of the filtration. It appears 
therefore that the decomposition of Doob Meyer depends more on the structure of the index set 
than on the filtration. 

A less general decomposition is based on the following, further uniform integrability 
condition for processes.
\begin{definition}
\label{definition uiv}
A stochastic process $\bar Y$ on $\bar{\Bbb A}$ is said to be of $P$ uniformly 
integrable variation for some $P\in\Prob(\F)$ if the collection
$\mathscr V(\bar Y)=\left\{\sum_{u\cap v^c\in d}\dabs{\bar Y_u-\condP{\bar Y_v}{\AU_u}}:d\in\DU\right\}$
is unformly $P$ integrable.
\end{definition}

\begin{corollary}
\label{corollary DM uiv}
Let $P\in\Prob(\F)$ and $\bar X$ be a positive, strong, $P$ supermartingale on $\bar{\Bbb A}$.
Then $\bar X$ is of $P$ uniformly integrable variation if and only if it decomposes as in
(\ref{DM}) with the increasing process $\bar A$ being, in addition, of $P$ uniformly integrable 
variation.
\end{corollary}

\begin{proof}
If (\ref{DM}) holds then $\mathscr V(\bar X)=\mathscr V(\bar A)$ so that $\bar X$ is of $P$
uniformly integrable variation if and only if so is $\bar A$. Assume then 
that $\bar X$ is of $P$ uniformly integrable variation and let $\bar m$ be the
finitely additive supermartingale associated to $\bar X$. Choose $\lambda\in\M^{uc}(\bar m)$ 
and let $\lP$ be its $P$ compensator. If $F\in\F$ then 
$$
\lambda(\bar\Pred^d_P(F))=\sum_{u\cap v^c\in d}P(\condP{F}{\AU_u}(X_u-X_v))
= P\left(\set{F}\sum_{u\cap v^c\in d}X_u-\condP{X_v}{\AU_u}\right)
\leq\sup_{V\in\DU}P(V\set{F})
$$
Thus, $\lrestr{\lP_{\OI}}{\F}\ll \lrestr{P}{\F}$ and $\bar m$ is thus of class $D_*$: (\ref{DM}) 
follows from Theorem \ref{theorem DM}. 
\end{proof}

The characterisation provided in Corollary \ref{corollary DM uiv} is less
satisfactory then it may appear at first sight. In fact the property involved
is significantly stronger than what considered in the classical setting.
In fact even increasing processes may fail to be of uniformly integrable
variation.

The special case of a linearly ordered index set is eventually considered.

\begin{corollary}
\label{corollary linear}
Let $m$ be a finitely additive supermartingale on $\Bbb A$ and assume that $U$ is linearly
ordered. Then the following are equivalent:
\begin{enumerate}
\item[(\textit{i})] $\M^{uc}(m)\neq\varnothing$;
\item[(\textit{ii})] $m$ is of class $D_*$;
\item[(\textit{iii})] there exists $P\in\Prob(\F)$ and a strong $P$ supermartingale $\bar X$ 
of uniformly integrable variation that meet (\ref{DM raw}).
\end{enumerate}
\end{corollary}

\begin{proof}
If $U$ is linearly ordered then each $d\in\DU$ may be taken to be of the form 
$\{u_n\cap v_n^c:u_n\leq u_{n+1}:n=1,\ldots,N\}$. Assume (\textit{i}) and choose 
$\lambda\in\M^{uc}$ and $P\in\Prob(\F,\lambda)$. If $F\in\F$ and $d\in\DU$ then, 
letting $M^*_d(F)=\sup_{\{1\leq n\leq N\}}\condP{F}{\AU_{u_n}}$ we have
\begin{equation}
\label{maximal}
\lambda(\Pd_P(F))\leq\eta\norm{\lambda}+\lambda(\Pd_P(F)>\eta)\leq\eta\norm{\lambda}
+\lambda_{\OI}(M^*_d(F)>\eta)
\end{equation}
Choose, by absolute continuity, $\delta$ such that $P(F)<\delta$ implies $\lambda_{\OI}(F)<\eta$.
By Doob maximal inequality if $P(F)<\eta\delta/3$ then $P(M^*_d(F)>\eta)\leq\delta$
and thus $\lambda(\Pd_P(F))\leq\eta(\norm{\lambda}+1)$. But then 
$\lambda_{\OI}^P(F)\leq\eta(\norm{\lambda}+1)$ and thus $\restr{\lambda_{\OI}^P}{\F}\ll\restr{P}{\F}$,
proving the implication (\textit{i})$\to$(\textit{ii}).
Assume (\textit{ii}) and consider $\lambda\in\M^{uc}(m)$, $P\in\Prob(\F,\lambda)$ and 
$\lambda^P\in\M^*(m)$. Let $\bar X$ be as in (\textit{iii}) and thus a strong $P$ supermartingale. 
Then,
\begin{equation}
\label{var to l}
P\left(F\sum_{u\cap v^c\in d}X_u-\condP{X_v}{\AU_u}\right)=
P\sum_{u\cap v^c\in d}\condP{F}{\AU_u}(X_u-X_v)=\lambda(\Pd_P(F))
\end{equation}
which implies, together with (\ref{maximal}), that $\bar X$ is of uniformly integrable variation.
Let $P$ and $\bar X$ be as in (\textit{iii}) and let $\bar m$ be the 
generated finitely additive supermartingale on $\bar{\Bbb A}$. Given that $\bar m$ 
is strong it is of class $D_0$, by Theorem \ref{theorem Doleans}. The restriction 
$\restr{\lambda}{\bar\Pred}$ of $\lambda\in\M(\bar m)$ to $\bar\Pred$ admits an 
extension to $\EU^*$ (still denoted by $\lambda$), defined by letting:
$$
\lambda(f)=\LIM_{d\in\DU}\lambda(\bar\Pred^d_P(f))
$$
Given that $\bar X$ is of uniformly integrable variation and in view of (\ref{var to l})
we conclude that $\lambda\in\M^{uc}(\bar m)$.
\end{proof}

\end{document}